\documentclass[11pt]{amsart}
\usepackage{amssymb}
\usepackage{amsmath, amscd}
\usepackage{amsthm}
\usepackage[table,xcdraw]{xcolor}
\usepackage{ulem}
\usepackage{nicematrix}
\usepackage{tikz}
\usepackage{circuitikz}
\usepackage[colorlinks=true, linkcolor=blue,urlcolor=blue]{hyperref}

\newtheorem{theorem}{Theorem}[section]

\newtheorem{proposition}[theorem]{Proposition}
\newtheorem{lemma}[theorem]{Lemma}

\newtheorem{corollary}[theorem]{Corollary}
\theoremstyle{definition}

\newtheorem{example}[theorem]{Example}

\newtheorem{remark}[theorem]{Remark}

\date\today

\begin{document}

\author[O. Hasanzadeh]{Omid Hasanzadeh}
\address{Department of Mathematics, Tarbiat Modares University, 14115-111 Tehran Jalal AleAhmad Nasr, Iran}
\email{o.hasanzade@modares.ac.ir; hasanzadeomiid@gmail.com}

\author[A. Moussavi]{Ahmad Moussavi}
\address{Department of Mathematics, Tarbiat Modares University, 14115-111 Tehran Jalal AleAhmad Nasr, Iran}
\email{moussavi.a@modares.ac.ir; moussavi.a@gmail.com}

\author[P. Danchev]{Peter Danchev}
\address{Institute of Mathematics and Informatics, Bulgarian Academy of Sciences, 1113 Sofia, Bulgaria}
\email{danchev@math.bas.bg; pvdanchev@yahoo.com}

\title[$U\sqrt{\Delta}$-rings]{Rings whose Non-Units are a Unit Multiple \\ of an Element from $\sqrt{\Delta(R)}$}

\begin{abstract}
This paper introduces and studies a new class of rings called {\it $U\sqrt{\Delta}$-rings}. A ring $R$ is $U\sqrt{\Delta}$ if every non-unit element can be written as the product of a unit and an element from $\sqrt{\Delta(R)}$, where $\sqrt{\Delta(R)}$ consists of elements some power of which lies in the special subring $\Delta(R)$. We establish certain basic properties of these rings and, concretely, prove that they are simultaneously indecomposable and Dedekind-finite. We also show that the polynomial ring $R[x]$ and the Laurent polynomial ring $R[x, x^{-1}]$ are never $U\sqrt{\Delta}$-rings, while the power series ring $R[[x]]$ inherits this property from $R$. Likewise, for left (right) Artinian rings, the conditions of being a $U\sqrt{\Delta}$-ring and a $UN$-ring are equivalent, as well as these two conditions are preserved for the full matrix ring $M_n(R)$ of size $n\geq 1$ over $R$. In addition, for a commutative ring $R$, $M_n(R)$ is a $U\sqrt{\Delta}$-ring exactly when $R$ is local. Furthermore, we characterize when a group ring $RG$ is a $U\sqrt{\Delta}$-ring showing that, for a locally solvable group $G$, this occurs precisely when $R$ is a $U\sqrt{\Delta}$-ring and $G$ is a locally finite $p$-group for some prime $p \in J(R)$.
\end{abstract}

\maketitle

\section{Introduction and Motivations}

Throughout this article, all rings are assumed to be associative with identity. In \cite{cfine}, C\v{a}lug\v{a}reanu and Lam introduced the concept of fine rings. A ring is called {\it fine} if every non-zero element can be written as the sum of a unit and a nilpotent element. This notion gives a new and proper class of simple rings, which strictly contains the class of all simple Artinian rings (that is, the full matrix rings over division rings). Fine rings and several related extensions have been studied intensively by many authors; see, for example, \cite{czuun, cziun, nzmax, zgen}. Paralleling this, in \cite{cun}, C\v{a}lug\v{a}reanu introduced the concept of a $UN$-{\it ring} as being a ring in which every non-unit element is the product of a unit and a nilpotent element. This class of rings has also been investigated by several authors; see, for instance, \cite{vun, zun}.

Additionally, V\'{a}mos introduced in \cite{v2g} the so-called {\it 2-good rings}, defined as those rings in which every element can be expressed as the sum of two units. In \cite{cun}, it was demonstrated that $UN$-rings with a 2-good identity are, actually, 2-good rings. Furthermore, a question was posed there to identify a proper subclass $\mathcal{C}$ of rings such that
\[
\{\text{$UN$-rings with a 2-good identity}\} \subset \mathcal{C} \subset \{\text{2-good rings}\}.
\]
In \cite[Example~1]{zun}, Zhou showed that if a ring $R$ possesses a 2-good identity and $R/J(R)$ is a $UN$-ring, then $R \in \mathcal{C}$, where $J(R)$ denotes the Jacobson radical of $R$ as the intersection of all maximal left ideals of $R$.

In this context, Saini and Udar further introduced and explored in \cite{sd} a new class of rings that is a common generalization of $UN$-rings, termed {\it $U\sqrt{J}$-rings}. A ring $R$ is called $U\sqrt{J}$ if every non-unit element of $R$ can be represented as a product of a unit and an element from $\sqrt{J(R)}$.

Following Wang and Chen \cite{wc}, this radical can be expressed as
\[
\sqrt{J(R)}=\{\, a\in R \mid a^n \in J(R)\ \text{for some } n \ge 1 \,\}.
\]
They established that the class of $U\sqrt{J}$-rings strictly lies between the class of $UN$-rings with 2-good identity and the class of 2-good rings, as they gave an affirmative answer to the question posed by C\v{a}lug\v{a}reanu in \cite{cun} as stated above.

Moreover, for the unit group $U(R)$ consisting of all units of the ring $R$, the focus of this research examination is presently the subring

\[
\begin{aligned}
J(R) \subseteq \Delta(R)=\{x\in R \mid x+u\in U(R)\ \forall\,u\in U(R)\}
\\[2mm]=
\{x\in R \mid 1-xu\in U(R)\ \forall\,u\in U(R)\}
\\[2mm]=
\{x\in R \mid 1-ux\in U(R)\ \forall\,u\in U(R)\},
\end{aligned}
\]
which was first noticed by Lam \cite[\text{Exercise }4.24]{lamm}, and more recently explored in-depth by Leroy and Matczuk \cite{2}. Their results, particularly \cite[\text{Theorems}3 \text{and} 6]{2}, establish that \(\Delta(R)\) represents the largest Jacobson radical subring of \(R\) stable under multiplication by all units (or quasi-invertible elements). Furthermore, one has that $\Delta(R) = J(T)$, where \(T\) denotes the subring of \(R\) generated by its units, and the condition \(\Delta(R)=J(R)\) is satisfied uniquely when \(\Delta(R)\) forms an ideal of \(R\).

Now, imitating \cite{Udar2}, \( R \) is called a {\it $UQ$-ring} if every non-unit element of \( R \) can be written as the product of a unit and a {\it quasi-regular} element, where an element \( a \in R \) is said to be {\it quasi-regular} provided \( a + 1 \) is a unit.

In the current work, we introduce $U\sqrt{\Delta}$-{\it rings}, which is a logical non-trivial generalization of both $U\sqrt{J}$ and $UN$-rings, as follows: These are the rings $R$ in which each non-unit element is a product of a unit from $R$ and an element from $\sqrt{\Delta(R)}$, where $$\sqrt{\Delta(R)} := \{ x \in R \mid x^n \in J(R) \text{ for some } n \in \mathbb{N} \}.$$

Accordingly, we have the following inclusions among the mentioned above classes of rings:
\[
\{UN\text{-rings}\} \subsetneqq \{U\sqrt{J}\text{-rings}\} \subseteq{\{U\sqrt{\Delta}\text{-rings}\}} \subsetneqq \{UQ\text{-rings}\}.
\]

In Section~\ref{sec2}, we first define the basic properties of the set $\sqrt{\Delta(R)}$ and use them to construct non-trivial examples of $U\sqrt{\Delta}$-rings. Notably, we show that every $U\sqrt{\Delta}$-ring is, simultaneously, indecomposable and Dedekind-finite (see Proposition~\ref{indecomposable}). We also illustrate that, for any ring $R$, the polynomial ring $R[x]$ and the Laurent polynomial ring $R[x,x^{-1}]$ are {\it not} $U\sqrt{\Delta}$-rings, whereas the power series ring $R[[x]]$ is a $U\sqrt{\Delta}$-ring if, and only if, $R$ itself is a $U\sqrt{\Delta}$-ring (see Example~\ref{laurent}, Proposition~\ref{3.2} and Corollary~\ref{power}).

In Section~\ref{sec3}, for left (resp., right) Artinian ring $R$, we show that $R$ is a $U\sqrt{\Delta}$-ring if, and only if, $R$ is a $UN$-ring if, and only if, $M_n(R)$ is a $UN$-ring if, and only if, $M_n(R)$ is a $U\sqrt{\Delta}$-ring (see Theorem~\ref{art}). In case of commutative former ring $R$, we prove that $M_n(R)$ is a $U\sqrt{\Delta}$-ring if, and only if, $R$ is a local ring (see Corollary~\ref{mm}).

In Section~\ref{sec4}, we present necessary and sufficient conditions for the group ring $RG$ to be a $U\sqrt{\Delta}$-ring. Specifically, we establish that, for a ring $R$ and a locally solvable group $G$, the group ring $RG$ is a $U\sqrt{\Delta}$-ring if, and only if, $R$ is a $U\sqrt{\Delta}$-ring and $G$ is a locally finite $p$-group for some prime $p \in J(R)$ (see Theorem~\ref{group ring} and Corollary~\ref{loc}).

\medskip

For completeness of the exposition and the reader's convenience, we recall some basic notations which will be frequently used below: for an arbitrary ring $R$, the set of nilpotent elements of $R$, the set of idempotent elements of $R$, and the center of $R$ are, respectively, designed by ${\rm Nil}(R)$, ${\rm Id}(R)$ and $C(R)$. The polynomial and formal power series rings on the variable $x$ over $R$ are denoted by $R[x]$ and $R[[x]]$, respectively. Recall that a ring $R$ is {\it weakly 2-primal} if ${\rm Nil}(R) = {\rm L}(R)$, where the latter symbol means the Levitzki radical of $R$. E.g., every commutative or every reduced ring is always weakly 2-primal.

We also set $M_n(R)$ to denote the ring of all $n\times n$ matrices over $R$.

\section{Examples and Basic Properties}\label{sec2}

Let $R$ be a ring. Put
$$\sqrt{\Delta(R)} := \{ x \in R \mid   x^n \in \Delta(R) \text{ for some } n \in \mathbb{N} \}.$$

Now, define

$$U\sqrt{\Delta(R)} := \{ x \in R \mid x = ab \text{ for some } a \in U(R) \text{ and } b \in \sqrt{\Delta(R)} \}.$$


We begin our study with the following technicality used freely in the sequel.


\begin{lemma}\label{property of J sharp}
For each ring \( R \), we have:

(1) If \( a \in \sqrt{\Delta(R)} \) and \( b \in U(R) \) with \( ab = ba \), then \( ab \in \sqrt{\Delta(R)} \).

(2) If \( a \in \sqrt{\Delta(R)} \), then \( 1 - a \in U(R) \).

(3) If \( a \in \sqrt{\Delta(R)} \cap C(R) \), then \( a \in \Delta(R) \).

(4) For every ideal \( I \subseteq J(R) \), we have \( \sqrt{\Delta(R/I)} = \sqrt{\Delta(R)}/I \).

(5) If \( R = \prod_{i=1}^k R_i \), then \( \sqrt{\Delta(R)} = \prod_{i=1}^k \sqrt{\Delta(R_i)} \).

(6) $\sqrt{\Delta(R)} \Delta(R) \subseteq \Delta(R)$.

(7) $\sqrt{\Delta(R)} + \Delta(R) \subseteq \sqrt{\Delta(R)}$.

(8) If $a\in U(\sqrt{\Delta}(R))\cap C(R)$, then $1-a\in U(R)$.

(9) If $a\in U(\sqrt{\Delta}(R))$, then, for each $u,v\in U(R)$, $uav\in U\sqrt{\Delta}(R)$.

(10) If $a\in \sqrt{\Delta(R)}$, then, for each $u\in U(R)$, $u^{-1}au\in \sqrt{\Delta(R)}$.

(11) If $ac\in\sqrt{\Delta(R)}$, where $c\in C(R)\cap U(R)$, then $a\in \sqrt{\Delta(R)}$.
\end{lemma}

\begin{proof}
(1) Suppose \( a \in \sqrt{\Delta(R)} \) and $b\in U(R)$ with \( ab=ba \). Thus, there exists \( n \in \mathbb{N} \) such that \( a^n \in \Delta(R) \), and so \( (ab)^n = a^nb^n \in \Delta(R) \) by \cite[Lemma 1(2)]{2}, whence \( ab \in \sqrt{\Delta(R)} \).

(2) Suppose \( a \in \sqrt{\Delta(R)} \). Thus, there exists \( n \in \mathbb{N} \) such that \( a^n \in \Delta(R) \). So, \[ (1-a)(1+a+ \cdots + a^{n-1}) = 1-a^n \in U(R) ,\] and hence \( 1-a \in U(R) \).

(3) Suppose \( a \in \sqrt{\Delta(R)} \cap C(R) \). Then, for every \( r \in U(R) \), it follows from (1) that \( ar \in \sqrt{\Delta(R)} \). So, in view of (2), \( 1-ar \in U(R) \), and thus \( a \in \Delta(R) \).

(4) This statement easily follows from $\Delta(R/I)=\Delta(R)/I$ (see \cite[Proposition 6(3)]{2}).

(5) Knowing that \( \Delta(R) = \prod_{i=1}^k \Delta(R_i) \) (see, e.g., \cite[Lemma 1(5)]{2}), the proof is straightforward.

(6) Let $a \in \sqrt{\Delta(R)}$ and $b \in \Delta(R)$. By part (2), we have $1-a \in U(R)$. Since $\Delta(R)$ is closed under multiplication by invertible elements, we obtain $b-ab=(1-a)b \in \Delta(R)$. Moreover, since $\Delta(R)$ is closed under addition, we conclude that $ab \in \Delta(R)$.

(7) Choose $a \in \sqrt{\Delta(R)}$ and $b \in \Delta(R)$. Then, there exists $k \in \mathbb{N}$ such that $a^k \in \Delta(R)$. In virtue of part (6) and the fact that $\Delta(R)$ is closed under addition, we deduce
\[
(a+b)^k = a^k + \sum_{\text{finite}} (\text{products of } k \text{ factors } a \text{ and } b) \in \Delta(R) + \Delta(R) \subseteq \Delta(R).
\]

(8) Write $a=uq$, where $u\in U(R)$ and $q\in \sqrt{\Delta(R)}$. Note that \[ qu = u^{-1}uqu = u^{-1}au  = a = uq .\] Now, $a\in \sqrt{\Delta(R)}$ holds via (1), so $1-a\in U(R)$ thanks to (2).

(9) Suppose $a = xy$ is a $U\sqrt{\Delta}$-decomposition, i.e., $x \in U(R)$ and $y \in \sqrt{\Delta(R)}$. Then, one writes that $uav = (uxv)(v^{-1}yv)$, which is also a $U\sqrt{\Delta}$-decomposition.

(10) This is clearly true, because $\Delta(R)$ remains invariant under multiplication by units (see \cite[Lemma 1(2)]{2}).

(11) Assuming that $ac\in \sqrt{\Delta(R)}$, there exists some $n\in \mathbb{N}$ such that $(ac)^n\in \Delta(R)$.
Since $c\in C(R)$, we get $(ac)^n=a^nc^n \in \Delta(R)$. Viewing \cite[Lemma 1(2)]{2}, the last implies that $a^n\in \Delta(R)$. Hence, $a\in \sqrt{\Delta(R)}$.
\end{proof}


Two types of commentaries now arise.

\begin{remark}\label{local}
For each ring \( R \), we have validity of the following four points:

(1) It is well known that $\Delta(R)\subseteq \sqrt{\Delta(R)}$ and ${\rm Nil}(R)\subseteq \sqrt{\Delta(R)}$;
however, the converse implications do {\it not} hold in general. To check that, set $S:=M_{2}(\mathbb{F}_{2})$,
$K=\mathbb{F}_{2}[[x]]$, and $L=S\times K$. Consider the element
$
A=\begin{pmatrix}
	1 & 1\\
	1 & 1
\end{pmatrix}\in S.
$
Then, one verifies that $A^{2}=0\in \Delta(S)$, which yields that $A\in \sqrt{\Delta(S)}$, while
$A\not\in \Delta(S)$. Moreover, $x\in J(K)\subseteq \sqrt{\Delta(K)}$, but
$x\not\in {\rm Nil}(K)$. Now, let $a:=(A,x)\in L$. Utilizing Lemma\ref{property of J sharp}(5), we obtain
$a\in \sqrt{\Delta(L)}$; however, $a\not\in J(L)$ and $a\not\in {\rm Nil}(L)$.

(2) \( \sqrt{\Delta(R)} \cap {\rm Id}(R) = \{0\} \). In fact, if \( e \in \sqrt{\Delta(R)} \cap {\rm Id}(R) \), then Lemma~\ref{property of J sharp}(2) ensures that \( 1-e \in {\rm Id}(R) \cap U(R) \), and so \( e=0 \).

(3) \( R \) is a local ring if, and only if, \( R=U(R) \cup \sqrt{\Delta(R)} \). Indeed, if \( R \) is local, then
one sees that \[ R=U(R) \cup J(R) \subseteq U(R) \cup \sqrt{\Delta(R)} \subseteq R ,\] and thus \( R=U(R) \cup \sqrt{\Delta(R)} \).

Conversely, suppose \( R=U(R) \cup \sqrt{\Delta(R)} \) and \( a \notin U(R) \). Then, \( a \in \sqrt{\Delta(R)} \), so that Lemma~\ref{property of J sharp}(2) employs to get that \( 1-a \in U(R) \), whence \( R \) is a local ring.

(4) \( \sqrt{\Delta(R)} \cap U(R) = \emptyset \). Indeed, if \( u \in \sqrt{\Delta(R)} \cap U(R) \), then there exists \( n \in \mathbb{N} \) such that \( u^n \in \Delta(R) \), and therefore \( 0=1-u^{-n}u^n \in U(R) \), which is an obvious contradiction.
\end{remark}

\begin{remark}\label{locall}
It is well known that $\sqrt{J(R)}\subseteq \sqrt{\Delta(R)}$; however, the converse implication does {\it not} hold in general. To verify that, with \cite[Ex 4.24]{lamm} at hand we choose $A$ to be any commutative domain which is not $J$-semi-simple (i.e., $J(A)$ is non-zero). Now, consider $R:=A[t]$. It is easy to see that $J(R)={\rm Nil}(R)=(0)$ and hence $\sqrt{J(R)}=(0)$. But, $\Delta(R)$ is non-zero and so $\sqrt{\Delta(R)}$ is non-zero, as expected, giving our claim.
\end{remark}

We now proceed by proving the following.

\begin{proposition}
Let \( a, b \in R \). If \( 1 - ab \in \sqrt{\Delta(R)} \), then \( 1 - ba \in \sqrt{\Delta(R)} \) holds if and only if \( a \in U(R) \).	
\end{proposition}

\begin{proof}
Let $a\in U(R)$. Thus, Lemma \ref{property of J sharp}(10) works to have $$1-ba=a^{-1}(1-ab)a \in \sqrt{\Delta(R)}.$$

Conversely, since $1 - ab, 1-ba \in \sqrt{\Delta(R)}$, we derive that $ab, ba\in U(R)$ and hence $a\in U(R)$.
\end{proof}

For a ring \( R \), we standardly say that \( {\rm Id}(R) \) (resp., \( U(R) \) and \( \sqrt{\Delta(R)} \)) is trivial whenever \( {\rm Id}(R) = \{0,1\} \) (resp., \( U(R) = \{1\} \) and \( \sqrt{\Delta(R)} = \{0\} \)).

\medskip

We now pay attention to prove the following principal result.

\begin{theorem}
Let \( R \) be a ring such that \( R = U(R) \cup \sqrt{\Delta(R)} \cup {\rm Id}(R) \). Then, exactly one of the following statements holds:
	
(1) \( R \) is a local ring;
	
(2) \( R \) is a Boolean ring;
	
(3) \( R \) is a non-abelian ring with \( {\rm char}(R) = 2 \).
\end{theorem}

\begin{proof}
We divide the argument into two main cases depending on the nature of $\sqrt{\Delta(R)}$.

\medskip
	
\textbf{Case 1:} Suppose $\sqrt{\Delta(R)} = \{0\}$. Then, we have $R = U(R) \cup Id(R)$. According to \cite[Theorem 2.4]{cc}, this forces that either $R$ is a division ring, and hence local, or $R$ is Boolean.

\medskip
	
\textbf{Case 2:} Assume now that $\sqrt{\Delta(R)} \neq \{0\}$. Let $d \in \sqrt{\Delta(R)} \setminus \{0\}$, and define $u := 1 + d$. Since $d \not= 0$, we find $u \in U(R) \setminus \{1\}$ implying that the unit group is non-trivial.

\medskip
	
\textbf{Subcase 2.1:} If ${\rm Id}(R) = \{0,1\}$, then we perceive $R = U(R) \cup \sqrt{\Delta(R)}$, which assures that $R$ is local owing to Remark \ref{local}(3).

\medskip
	
\textbf{Subcase 2.2:} Suppose all three subsets $U(R)$, $\sqrt{\Delta(R)}$ and ${\rm Id}(R)$ are non-trivial. Let $e \in {\rm Id}(R) \setminus \{0,1\}$, and again write $u = 1 + d \in U(R) \setminus \{1\}$. We show now that $2 \not\in U(R)$.
	
Assume, for contradiction, that $2 \in U(R)$. Consider the element $2e$. If $2e \in {\rm Id}(R)$, then $(2e)^2 = 2e$ insures $e = 0$, contradicting the choice of $e$. Thus, $2e$ must lie in $\sqrt{\Delta(R)}$. But, viewing Lemma \ref{property of J sharp}(11), this gives $e \in \sqrt{\Delta(R)}$, contradicting $e \in {\rm Id}(R)$. So, $2 \not\in U(R)$, as promised.
	
Therefore, $2$ has to lie in either ${\rm Id}(R)$ or $\sqrt{\Delta(R)}$. If $2^2 = 2$, then $2 = 0$, i.e., ${\rm char}(R) = 2$. Alternatively, if $2 \in \sqrt{\Delta(R)}$, then $3 = 1 + 2 \in U(R)$.
	
We now examine the element $3e$. Clearly, $3e \not\in U(R)$ as $e \neq 1$. If $3e \in \sqrt{\Delta(R)}$, then Lemma \ref{property of J sharp}(11) again guarantees $e = 0$, a contradiction. Thus, $3e \in {\rm Id}(R)$ and, since $(3e)^2 = 3e$, we get $6e = 0$, whence $2e = 0$. A similar computation using $1 - e$ leads to $2(1 - e) = 0$, so that $2 = 0$ and again ${\rm char}(R) = 2$.
	
To conclude argumentation, we show that $R$ is not abelian. Suppose on the contrary that it is abelian, then $ue = eu$. However, $ue \not\in U(R)$ and $ue \not\in \sqrt{\Delta(R)}$. So, $ue \in {\rm Id}(R)$ means $(ue)^2 = ue$, i.e., $ue = e$ and, likewise $u(1 - e) = 1 - e$, thus inferring $$u = ue + u(1 - e) = e + (1 - e) = 1,$$ and contradicting $u \not= 1$, as asked. Hence, $R$ is a non-abelian ring with ${\rm char}(R) = 2$, concluding the assertion.
\end{proof}

The next claims are pretty elementary, but worthy of mentioning being useful in our further considerations.

\begin{example}

(1) Every element in $\sqrt{\Delta(R)}$ is a $U\sqrt{\Delta}$-element.
		
(2) Every element in ${\rm Nil}(R)$ is a $U\sqrt{\Delta}$-element.
		
(3) If $R$ is a $UU$-ring, that is, $U(R)=1+{\rm Nil}(R)$, then $R$ is $U\sqrt{\Delta}$ if and only if $R$ is $UN$.
		
(4) Every simple Artinian ring is a $U\sqrt{\Delta}$-ring.
\end{example}


The following statement shows that each $U\sqrt{\Delta}$-ring is both indecomposable and Dedekind-finite.

\begin{proposition}\label{indecomposable}
Let \( R \) be a \( U\sqrt{\Delta} \)-ring. Then,

(1) \( R \) does not have a non-trivial central idempotent. In particular, \( R \) is an indecomposable ring.

(2) \( R \) is Dedekind-finite.
\end{proposition}

\begin{proof}
(1) Assume \( e \in {\rm Id}(R) \cap C(R) \). Working with Lemma~\ref{property of J sharp}(8), we find that \( 1 - e \in {\rm Id}(R) \cap U(R) \), hence \( e = 0 \), as needed.

(2) Suppose on the opposite that $R$ is not a Dedekind-finite ring. Thus, there exist elements $a, b \in R \setminus U(R)$ such that $ab = 1$ but $ba \neq 1$. Assume that $a = vp$ and $b = uq$ are two respective $U\sqrt{\Delta}$-decompositions for $a$ and $b$. Let $n \in \mathbb{N}$ be the smallest positive integer such that $q^n \in \Delta(R)$. So, because of the relation $U(R)\Delta(R) \subseteq \Delta(R)$ and with the help of Lemma \ref{property of J sharp}(6), we obtain
\[
q^{n-1} = abq^{n-1} = vpuq^n \in \Delta(R),
\]
which contradicts the minimality of $n \in \mathbb{N}$, as wanted.
\end{proof}


Our next helpful technical claims are the following.

\begin{lemma}\label{facto}
Let \( R \) be a ring, and let \( I \) be an ideal of $R$. If $R$ is a $U\sqrt{\Delta}$-ring, then $R/I$ is also a $U\sqrt{\Delta}$-ring. Conversely, if \( I \subseteq J(R) \) and \( R/I \) is a $U\sqrt{\Delta}$-ring, then $R$ is a $U\sqrt{\Delta}$-ring.
\end{lemma}

\begin{proof}
Assume \( \bar a=a + I \notin U(R/I) \). Thus, \( a \notin U(R) \). Therefore, there exist \( u \in U(R) \) and \( q \in \sqrt{\Delta(R)} \) such that \( a = uq \).  Evidently, \( \bar{u} \in U(R/I) \) and applying Lemma~\ref{property of J sharp}(4), we see that $\bar q \in \sqrt{\Delta(R/I)}$ and \(\bar{a} = \bar{u} \bar{q} \). Therefore, $R/I$ is a $U\sqrt{\Delta}$-ring as well.

Reciprocally, suppose that \( I \subseteq J(R) \), \( R/I \) is a $U\sqrt{\Delta}$-ring and \( a \notin U(R) \). Then, \( \bar{a} \notin U(R/I) \), so there exist \( \bar{q} \in \sqrt{\Delta(R/I)} \) and \( \bar{u} \in U(R/I) \) such that \( \bar{a} = \bar{u}\bar{q} \). Hence, \( a - uq \in I \subseteq J(R) \), so that there exists some \( j \in J(R) \) with \( a = uq + j = u(q + u^{-1}j)\). Apparently, \( u \in U(R) \), and Lemma \ref{property of J sharp}(4) applies to get that \( q \in \sqrt{\Delta(R)} \). Now, Lemma~\ref{property of J sharp}(7) is applicable to write that \( q + u^{-1}j \in \sqrt{\Delta(R)} \), and so $a\in U\sqrt{\Delta}(R)$, as desired.
\end{proof}


\begin{proposition}\label{center UJ}
Let $R$ be a $U\sqrt{\Delta}$-ring. Then, $C(R)$ is a local ring. In particular, if $R$ is commutative, then $R$ is a $U\sqrt{\Delta}$-ring if and only if $R$ is a local ring.
\end{proposition}

\begin{proof}
Assume that $a\in C(R)$ with $a\not\in U(C(R))$. Thus, $a$ is a non-unit element of $R$, and hence $a\in C(R)\cap U\sqrt{\Delta}(R)$. In accordance with Lemma~\ref{property of J sharp}(8), we write
\[
1-a\in U(R)\cap C(R)\subseteq U(C(R)).
\]
Therefore, $C(R)$ is a local ring, as stated.
\end{proof}


Some more concrete exhibitions on the established above assertions are these:

\begin{example}
$\mathbb{Z}_n$ is a $U\sqrt{\Delta}$-ring if and only if $n$ is a power of a prime number.
\end{example}

\begin{example}
Let $R = \mathbb{Z}_4 \times \mathbb{Z}_4$. Although $\mathbb{Z}_4$ is a $U\sqrt{\Delta}$-ring, the ring $R$ is itself {\it not} $U\sqrt{\Delta}$. Indeed, a simple technical check shows that the non-unit element $(2,3)$ cannot be expressed as the product of a unit and an element from $\sqrt{\Delta(R)}$.
\end{example}

\begin{remark}
(1) In general, the triangular matrix rings are {\it not} $U\sqrt{\Delta}$-rings; in fact, for instance, after a plain inspection $T_2(\mathbb{Z}_2)$ manifestly does {\it not} satisfy this property.

(2) A subring of a $U\sqrt{\Delta}$-ring need {\it not} preserve the same property. For example, while $M_2(\mathbb{Z}_2)$ is a $U\sqrt{\Delta}$-ring in conjunction with Corollary \ref{mm} listed below, its subring $T_2(\mathbb{Z}_2)$ is surely not.
\end{remark}

Let $R$ be a ring and $M$ a bi-module over $R$. The trivial extension of $R$ and $M$ is defined as
\[ T(R, M) = \{(r, m) : r \in R \text{ and } m \in M\}, \]
with addition defined componentwise and multiplication defined by
\[ (r, m)(s, n) = (rs, rn + ms). \]
Note that the trivial extension $T(R, M)$ is definitely isomorphic to the subring
\[ \left\{ \begin{pmatrix} r & m \\ 0 & r \end{pmatrix} : r \in R \text{ and } m \in M \right\} \]
of the formal $2 \times 2$ matrix ring $\begin{pmatrix} R & M \\ 0 & R \end{pmatrix}$, and also $T(R, R) \cong R[x]/(x^2)$. We, likewise, notice that the set of units of the trivial extension $T(R, M)$ is
\[ U(T(R, M)) = T(U(R), M). \]
Moreover, referring to \cite{kara}, we can write
\[ \Delta(T(R, M)) = T(\Delta(R), M). \]

\medskip

We are now in a position to prove the following.

\begin{proposition}\label{4.5}
Suppose $R$ is a ring and $M$ is a bi-module over $R$. Then, the following items are true:
	
(1) The trivial extension $T(R, M)$ is a $U\sqrt{\Delta}$-ring if and only if $R$ is a $U\sqrt{\Delta}$-ring.
	
(2) For $n \geq 2$, the quotient-ring $R[x; \alpha]/(x^n)$ is a $U\sqrt{\Delta}$-ring if and only if $R$ is a $U\sqrt{\Delta}$-ring.
	
(3) For $n \geq 2$, the quotient-ring $R[[x; \alpha]]/(x^n)$ is a $U\sqrt{\Delta}$-ring if and only if $R$ is a $U\sqrt{\Delta}$-ring.
\end{proposition}

\begin{proof}
(1) Put $A:=T(R, M)$, and consider $I:=T(0, M)$. It is not too difficult to see that $I\subseteq J(A)$ such that $A/I \cong R$. So, the statement follows directly from Proposition~\ref{facto}.
	
(2) Put $A:=R[x; \alpha]/(x^n)$. Considering the ideal $I:=(x)/(x^n)$ of $A$, we verify that $I\subseteq J(A)$ with $A/I \cong R$. So, the statement follows immediately from Proposition~\ref{facto}.
	
(3) Knowing that the isomorphism $$R[x; \alpha]/(x^n) \cong R[[x; \alpha]]/(x^n)$$ is fulfilled, the statement follows at once from point (2).
\end{proof}

As an immediate consequence, we yield:

\begin{corollary}\label{4.6}
Let $R$ be a ring. Then, the following are equivalent:
	
(1) $R$ is a $U\sqrt{\Delta}$-ring.
	
(2) For $n \geq 2$, the ring $R[x]/(x^n)$ is a $U\sqrt{\Delta}$-ring.
	
(3) For $n \geq 2$, the ring $R[[x]]/(x^n)$ is a $U\sqrt{\Delta}$-ring.
\end{corollary}

Suppose $R$ is a ring and $M$ is a bi-module over $R$. Putting $$DT(R,M) := \{ (a, m, b, n) | a, b \in R, m, n \in M \}$$ with addition defined componentwise and multiplication defined by $$(a_1, m_1, b_1, n_1)(a_2, m_2, b_2, n_2) = (a_1a_2, a_1m_2 + m_1a_2, a_1b_2 + b_1a_2, a_1n_2 + m_1b_2 + b_1m_2 +n_1a_2),$$ we then see that $DT(R,M)$ is a ring which is isomorphic to $T(T(R, M),  T(R, M))$.

Besides, we have $$DT(R, M) =
\left\{\begin{pmatrix}
	a &m &b &n\\
	0 &a &0 &b\\
	0 &0 &a &m\\
	0 &0 &0 &a
\end{pmatrix} |  a,b \in R, m,n \in M\right\}.$$ We now establish the following isomorphism of rings as follows: the map $$R[x, y]/(x^2, y^2) \rightarrow DT(R, R)$$ is defined by $$a + bx + cy + dxy \mapsto
\begin{pmatrix}
	a &b &c &d\\
	0 &a &0 &c\\
	0 &0 &a &b\\
	0 &0 &0 &a
\end{pmatrix},$$
and it is easy to check that it is a ring isomorphism.

\medskip

Two direct consequences are the following.

\begin{corollary}
Let $R$ be a ring and $M$ a bi-module over $R$. Then, the following assertions are equivalent:
	
(1) $R$ is a $U\sqrt{\Delta}$-ring ring.
	
(2) $DT(R, M)$ is a $U\sqrt{\Delta}$-ring.
	
(3) $DT(R, R)$ is a $U\sqrt{\Delta}$-ring.
	
(4) $R[x, y]/(x^2, y^2)$ is a $U\sqrt{\Delta}$-ring.
\end{corollary}

Assuming that
$$
L_n(R) := \left\{
\begin{pmatrix}
	0 & \cdots & 0 & a_1 \\
	0 & \cdots & 0 & a_2 \\
	\vdots & \ddots & \vdots & \vdots \\
	0 & \cdots & 0 & a_n
\end{pmatrix}
\in {\rm T}_n(R) \,\mid|\, a_1,\dots,a_n \in R \right\}
$$
and
$$
S_n(R) := \left\{(a_{ij}) \in {\rm T}_n(R) \mid a_{11} = \cdots = a_{nn} \right\},
$$
it is not so hard to verify that the mapping $\varphi: S_n(R) \to T(S_{n-1}(R), L_{n-1}(R))$, defined by
$$
\begin{pmatrix}
	a_{11} & a_{12} & \cdots & a_{1n} \\
	0 & a_{11} & \cdots & a_{2n} \\
	\vdots & \vdots & \ddots & \vdots \\
	0 & 0 & \cdots & a_{11}
\end{pmatrix}
\mapsto
\begin{pmatrix}
	a_{11} & a_{12} & \cdots & a_{1,n-1} & 0 & \cdots & 0 & a_{1n} \\
	0 & a_{11} & \cdots & a_{2,n-1} & 0 & \cdots & 0 & a_{2n} \\
	\vdots & \vdots & \ddots & \vdots & \vdots & \ddots & \vdots & \vdots \\
	0 & 0 & \cdots & a_{11} & 0 & \cdots & 0 & a_{n-1,n} \\
	0 & 0 & \cdots & 0 & a_{11} & a_{12} & \cdots & a_{1,n-1} \\
	0 & 0 & \cdots & 0 & 0 & a_{11} & \cdots & a_{2,n-1} \\
	\vdots & \vdots & \ddots & \vdots & \vdots & \vdots & \ddots & \vdots \\
	0 & 0 & \cdots & 0 & 0 & 0 & \cdots & a_{11}
\end{pmatrix},$$
illustrates that $$S_n(R) \cong T(S_{n-1}(R), L_{n-1}(R)).$$

This isomorphism serves as a key tool for studying the ring $S_n(R)$, because by examining the trivial extension and employing induction on $n$, we can substantially generalize the results to $S_n(R)$.


\begin{corollary}
For $n \ge 2$, $S_n(R)$ is a a $U\sqrt{\Delta}$-ring if and only if $R$ is a $U\sqrt{\Delta}$-ring.
\end{corollary}


We now continue our work with results on some polynomial ring extensions.

\begin{example}\label{laurent}
For any ring $R$, both the Laurent polynomial ring $R[x, x^{-1}]$ and the polynomial ring $R[x]$ are {\it not} $U\sqrt{\Delta}$-ring.
\end{example}

\begin{proof}
Suppose on the reverse that $R[x]$ is a $U\sqrt{\Delta}$-ring. Since $x \not\in U(R[x])$ and it is a central element, Lemma~\ref{property of J sharp}(8) enables us to write that $1-x \in U(R[x])$, a contradiction.

Now, suppose on the reciprocality that $R[x, x^{-1}]$ is a $U\sqrt{\Delta}$-ring. Since $1+x+x^2 \not\in U(R[x, x^{-1}])$ and it is a central element, one writes that $$-x(1+x)=1-(1+x+x^2) \in U(R[x, x^{-1}]),$$ thus enabling that $1+x \in U(R[x, x^{-1}])$, which is an obvious contradiction.
\end{proof}

We are now ready to prove the following slightly unexpected equivalence.

\begin{proposition}\label{3.2}
A ring $R[[x; \alpha]]$ is $U\sqrt{\Delta}$ if and only if so is $R$.
\end{proposition}

\begin{proof}
Setting $I:= R[[x; \alpha]]x$, we see that $I$ is an ideal of $R[[x; \alpha]]$. Noticing that the equality $J(R[[x; \alpha]])=J(R)+I$ is valid, it must be that $I\subseteq J(R[[x; \alpha]])$. And since $R[[x; \alpha]]/I\cong R$ holds, we apply Lemma \ref{facto} to conclude the claim, as asserted.
\end{proof}

A quick consequence is the following.

\begin{corollary}\label{power}
A ring $R[[x]]$ is $U\sqrt{\Delta}$ if, and only if, so is $R$.
\end{corollary}


Now, according to the above example, the polynomial ring \( R[x] \) is {\it not} a \( U\sqrt{\Delta} \)-ring. However, a natural question arises: {\it What is the explicit form of the elements in \( U\sqrt{\Delta}(R[x]) \)}?  The next result addresses this point for a more general class of rings, namely the skew polynomial ring \( R[x; \alpha] \).

Let \( R \) be a ring and let \( \alpha \) be an endomorphism of \( R \). Recall that \( R \) is called \textit{\(\alpha\)-compatible} provided that, for all \( a, b \in R \), we have $ab = 0$ if and only if $a\alpha(b) = 0$.

\begin{proposition}
Let $R$ be an $\alpha$-compatible ring. Then, $R$ is a weakly $2$-primal ring if and only if
$$U\sqrt{\Delta}(R[x; \alpha]) = U\sqrt{\Delta}(R) + {\rm L}(R)[x; \alpha]x.$$
\end{proposition}

\begin{proof}
First, assume that $$U\sqrt{\Delta}(R[x; \alpha]) = U\sqrt{\Delta}(R) + {\rm L}(R)[x; \alpha]x$$ is true. If $a^n=0$, then \cite[Lemma~2.1]{chen} allows to write that $a\alpha(a)\cdots\alpha^{n-1}(a)=0$, which shows that $(ax)^n=0$. Therefore, $$ax \in {\rm Nil}(R[x; \alpha]) \subseteq \sqrt{\Delta(R[x; \alpha])} \subseteq U\sqrt{\Delta}(R[x; \alpha]) = U\sqrt{\Delta}(R) + {\rm L}(R)[x; \alpha],$$ and consequently $a \in {\rm L}(R)$ yielding ${\rm Nil}(R)= {\rm L}(R)$, as requested.

Conversely, assume now that $R$ is a weakly $2$-primal ring, and establish three claims.

\medskip

\textbf{Claim (1):} $\Delta(R[x; \alpha]) = \Delta(R) + {\rm L}(R)[x; \alpha]x$.

\medskip

Write $f=\sum_{i=0}^n a_ix^i \in \Delta(R[x; \alpha])$. Then, for $u \in U(R)$, we have $1-uf \in U(R[x; \alpha])$, so \cite[Corollary~2.14]{chen} is a guarantor that $a_0 \in \Delta (R)$ and, for each $1 \leq i \leq n$, $a_i \in {\rm Nil}(R) ={\rm L}(R)$. Thus, $\Delta(R[x; \alpha]) \subseteq \Delta(R) + {\rm L}(R)[x; \alpha]x$.

Conversely, let $f=\sum_{i=0}^n a_ix^i \in \Delta(R) + {\rm L}(R)[x; \alpha]x$ and $U=\sum_{i=0}^n u_ix^i \in U(R[x; \alpha])$. Utilizing \cite[Corollary~2.14]{chen}, $u_0 \in U(R)$ and, for each $1 \leq i \leq n$, $u_i \in {\rm Nil}(R) ={\rm L}(R)$. Since ${\rm Nil}(R)={\rm L}(R)$ is an ideal, it is readily to show that $1-uf \in U(R)+{\rm L}(R)[x; \alpha]x$. So, an application of \cite[Corollary~2.14]{chen} gives that $1-uf \in U(R[x; \alpha])$, which means $f\in \Delta(R[x; \alpha])$.

\medskip

\textbf{Claim (2):} $\sqrt{\Delta(R[x; \alpha])} = \sqrt{\Delta(R)} + {\rm L}(R)[x; \alpha]x$.

\medskip

Write $f=\sum_{i=0}^n a_ix^i \in \sqrt{\Delta(R[x; \alpha])}$. Then, there exists $m \in \mathbb{N}$ such that $f^m \in \Delta(R[x; \alpha])$. By Claim (1), we have $a_0 \in \sqrt{\Delta(R)}$ and $$a_n\alpha^{n}(a_n)\cdots \alpha^{n(m-1)}(a_n) \in {\rm L}(R)={\rm Nil}(R),$$ so with \cite[Lemma~2.1]{chen} in mind $a_n \in {\rm Nil}(R)={\rm L}(R)$.

Now, we show that if $n\geq 2$, then $a_{n-1} \in {\rm L}(R)$. To that goal, set $f:=g-a_nx^n$. Then, $f^m=g^m+q$, where $q \in R[x; \alpha]$. Note that the coefficients of $q$ can be written as sums of monomials in $a_i$ and $\alpha^{k}(a_j)$, where $a_i, a_j \in \{a_0, a_1, \ldots, a_n\}$, $k \geq 0$ is a positive integer, and each monomial contains $a_n$. Since $a_n \in {\rm Nil}(R)$ and ${\rm Nil}(R)$ is an ideal, we conclude that $q \in {\rm Nil}(R)[x; \alpha]x$. Moreover, by Claim (1), ${\rm Nil}(R)[x; \alpha]x \subseteq \Delta(R[x; \alpha])$. Therefore, $g^m \in \Delta(R[x; \alpha])$. Again exploiting Claim (1), we detect that $$a_{n-1}\alpha^{n-1}(a_{n-1})\cdots\alpha^{(n-1)(m-1)}(a_{n-1}) \in {\rm L}(R)={\rm Nil}(R),$$ so an utilization of \cite[Lemma~2.1]{chen} leads to $a_{n-1} \in {\rm Nil}(R)={\rm L}(R)$. Continuing this process, we can show that, for each $1 \leq i \leq n$, $a_i \in {\rm L}(R)$.

Conversely, let $f=\sum_{i=0}^n a_ix^i \in \sqrt{\Delta(R)} + {\rm L}(R)[x; \alpha]x$. Then, there exists $m \in \mathbb{N}$ such that $a_0^m \in \Delta(R)$. A simple calculation shows that $f^m \in \Delta(R[x; \alpha])$.

\medskip

\textbf{Claim (3):} $U\sqrt{\Delta}(R[x; \alpha]) = \sqrt{\Delta}(R) + {\rm L}(R)[x; \alpha]x$.

\medskip

Write $f=\sum_{i=0}^n a_ix^i \in U\sqrt{\Delta}(R[x; \alpha])$. Then, there exist $u=\sum_{i=0}^n u_ix^i \in U(R[x; \alpha])$ and $q=\sum_{i=0}^n q_ix^i \in \sqrt{\Delta(R[x; \alpha])}$ such that $f=uq$. Bearing into account \cite[Corollary~2.14]{chen}, one infers that $u_0 \in U(R)$ and, for each $1 \leq i \leq n$, $u_i \in {\rm Nil}(R) ={\rm L}(R)$. Also, by Claim (2), $q_0 \in \sqrt{\Delta}(R)$ and, for each $1 \leq i \leq n$, $q_i \in {\rm Nil}(R) ={\rm L}(R)$. A plain computation demonstrates that $f\in \sqrt{\Delta}(R) + {\rm L}(R)[x; \alpha]x$.

Conversely, let $f \in \sqrt{\Delta}(R) + {\rm L}(R)[x; \alpha]x$. Then, there exist $u \in U(R)$, $q \in \sqrt{\Delta(R)}$ and, for each $1 \leq i \leq n$, $l_i \in {\rm L}(R)$ such that $f=uq+\sum_{i=1}^{n}l_ix^i$. Factoring out $u$ and using Claim (2), we conclude $$f=u(q+u^{-1}l_1x + \cdots u^{-1}l_nx^n) \in U\sqrt{\Delta}(R[x; \alpha]),$$ as required, thus giving the claimed equality.
\end{proof}

We are now focussing on a next series of technical preliminaries. Before doing that, we say that the rings $R$ for which each non-unit element is uniquely written as a product of a unit from $R$ and an element from $\sqrt{\Delta(R)}$ are {\it uniquely $U\sqrt{\Delta}$}, and those for which each non-unit element is written as a commuting product of a unit from $R$ and an element from $\sqrt{\Delta(R)}$ are {\it strongly $U\sqrt{\Delta}$}. Thereby, we succeed to show validity of the following necessary and sufficient condition.

\begin{lemma}
Let $R$ be a ring. Then,

(1) $R$ is uniquely $U\sqrt{\Delta}$ if and only if $R$ is a division ring.

(2) $R$ is strongly $U\sqrt{\Delta}$ if and only if $R$ is local.
\end{lemma}

\begin{proof}
(1) Suppose $d \in \sqrt{\Delta(R)}$. Then, we may write
\[
			d = 1 \cdot d = (1 - d)(1 - d)^{-1}d.
\]
Since $d$ and $(1 - d)^{-1}$ commute one another, Lemma \ref{property of J sharp}(1) is working to get $(1 - d)^{-1}d \in \sqrt{\Delta(R)}$, and hence
	\[
			1 = 1 - d \Rightarrow d = 0.
	\]
Therefore, $\sqrt{\Delta(R)} = 0$, and so $R$ is a division ring. The converse is rather clear, so we omit details.
			
(2) Suppose $R$ is strongly $U\sqrt{\Delta}$. Let $x \notin U(R)$. Then, $x = ud = du$ for some $u \in U(R)$ and $d \in \sqrt{\Delta(R)}$. Playing with Lemma \ref{property of J sharp}(1), we arrive at $x \in \sqrt{\Delta(R)}$, and thus $R$ is local in virtue of Remark \ref{local}(3).
			
Conversely, assume $R$ is local. Thus, again in view of Remark \ref{local}(3), all non-units are in $\sqrt{\Delta(R)}$. Letting $d \in \sqrt{\Delta(R)}$, then $d = 1 \cdot d$ is a strongly $U\sqrt{\Delta}$ decomposition. Finally, $R$ is strongly $U\sqrt{\Delta}$, as claimed.
\end{proof}

The following property is also worthwhile to be recorded.

\begin{proposition}
Any $U\sqrt{\Delta}$-ring $R$ is left-right symmetric.	
\end{proposition}

\begin{proof}
Let \( r \in R \) be a \( U\sqrt{\Delta} \)-element. Then, there exist \( u \in U(R) \) and \( d \in \sqrt{\Delta(R)} \) such that $r = ud$. Thus, we have $r = (u d u^{-1}) u$, where \( u d u^{-1} \in \sqrt{\Delta(R)} \) by
Lemma~\ref{property of J sharp}(10), and \( u \in U(R) \).
\end{proof}

We call a subring $S$ of a ring $R$ {\it rationally closed} whenever $U(R) \cap S = U(S)$.

\begin{proposition}\label{rationally closed}
Let $S \subseteq R$ be a rationally closed subring of $R$. Then, $\sqrt{\Delta(R)} \cap S \subseteq \sqrt{\Delta(S)}$.
\end{proposition}

\begin{proof}
Let $r\in \sqrt{\Delta(R)}\cap S$. So, there exists some $n\in \mathbb{N}$ such that $r^n\in \Delta(R)\cap S$.
Moreover, for every $u\in U(R)$, we have $1+r^n u\in U(R)$. Hence, for every $v\in U(S)$, it follows that
\[
1+r^n v\in U(R)\cap S = U(S).
\]
Thus, $r^n\in \Delta(S)$, and therefore $r\in \sqrt{\Delta(S)}$, as pursued.
\end{proof}

\begin{lemma}\label{center}
Let $R$ be a $U\sqrt{\Delta}$-ring. Then $C(R)$ is also a $U\sqrt{\Delta}$-ring.
\end{lemma}

\begin{proof}
Choose $x \in C(R)$ and suppose $x \not\in U(C(R))$ which amounts to $x \not\in U(R)$. Since $R$ is a $U\sqrt{\Delta}$ ring, we can write $x = ud$ for some $u \in U(R)$ and $d \in \sqrt{\Delta(R)}$.
	
Now,
	\[
	d = u^{-1}x = xu^{-1} \Rightarrow x = du = ud.
	\]
Arranging Lemma~\ref{property of J sharp}(1), we have $x = ud \in \sqrt{\Delta(R)}$. But, Proposition~\ref{rationally closed} tells us that $C(R) \cap \sqrt{\Delta(R)} \subseteq \sqrt{\Delta(C(R))}$, forcing $x \in \sqrt{\Delta(C(R))}$. Thus, $x \not\in U(C(R)) \Rightarrow x \in \sqrt{\Delta(C(R))}$, and Remark \ref{local}(3) suggests that $C(R)$ is a local ring, and therefore a $U\sqrt{\Delta}$-ring, as formulated.
\end{proof}

\begin{lemma}
Let $R$ be a $U\sqrt{\Delta}$ ring and let $n \in \mathbb{Z}$. Then, either $n\cdot 1 \in U(R)$ or $n\cdot 1 \in \Delta(R)$.
\end{lemma}

\begin{proof}
Assuming $n \not\in U(R)$, there exist $u \in U(R)$ and $d \in \sqrt{\Delta(R)}$ such that $n = ud = du$. Knowing Lemma \ref{property of J sharp}(1), we conclude that $n \in \sqrt{\Delta(R)}$. Therefore, Lemma \ref{property of J sharp}(3) informs us that $n \in \Delta(R)$.
\end{proof}

We now come to the following three curious assertions.

\begin{proposition}
Let $R$ be a ring. If $R/J(R)$ is $UN$, then $R$ is $U\sqrt{\Delta}$.
\end{proposition}

\begin{proof}
For any $a \in R$, consider $\bar a:=a + J(R) \in R/J(R)=\bar R$. Supposing that $\bar a$ is a non-unit in $\bar R$, then we can write $\bar a = \bar u \bar q$, where $\bar u \in U(\bar R)$ and $\bar q \in Nil(\bar R)$. So, $a = uq + j$ for some $j \in J(R)$. Thus, $a = u(q + u^{-1}j)$ with $u \in U(R)$ and $q \in {\rm Nil}(R)$. Note that
	\[
	q + u^{-1}j \in Nil(R) + J(R) \subseteq \sqrt{\Delta(R)}.
	\]
This shows that $q + u^{-1}j \in \sqrt{\Delta(R)}$. Hence, $a$ is of the form $u \cdot d$ for some $u \in U(R)$ and $d \in \sqrt{\Delta(R)}$, i.e., $a$ is $U\sqrt{\Delta}$. Therefore, $R$ is a $U\sqrt{\Delta}$ ring, as asserted.
\end{proof}

\begin{proposition}
Let $R$ be a strongly $\pi$-regular ring with only trivial idempotents. Then, $R$ is a $U\sqrt{\Delta}$ ring.
\end{proposition}

\begin{proof}
Assume that $R$ is strongly $\pi$-regular. According to \cite[Proposition 2.6]{bur}, $R$ is strongly clean. Now, since $R$ is strongly clean and has only trivial idempotents (that are, $0$ and $1$), it follows at once that $R$ is a local ring. But, as every local ring is $U\sqrt{\Delta}$, the result sustains.
\end{proof}

\begin{proposition}\label{good}
If $R$ is a $U\sqrt{\Delta}$-ring with a 2-good identity, then $R$ is a 2-good ring.
\end{proposition}

\begin{proof}
Let $a \in R$ be a non-unit element. Then, we can write $a = ud$, where $u \in U(R)$ and $d \in \sqrt{\Delta(R)}$. So,
	\[
	a = u(1 + d) - u = uv - u,
	\]
where $v := 1 + d \in U(R)$. By hypothesis, we have $1 = u + v$. So,
	\[
	u = u \cdot 1 = u^2 + uv,
	\]
where $u^2, uv\in U(R)$. Therefore, $R$ is a 2-good ring, as claimed.
\end{proof}

Combining \cite[Example~3.6]{sd} and Proposition \ref{good}, we obtain the following strict inclusions among well-known classes of rings:
\begin{align*}
	\{UN\text{-rings with 2-good identity}\} &\subsetneqq \{U\sqrt{J}\text{-rings with 2-good identity}\} \\
	&\subseteq {\{U\sqrt{\Delta}\text{-rings with 2-good identity}\}} \\
	&\subsetneqq \{2\text{-good rings}\}.
\end{align*}


\section{Matrix Rings over $U\sqrt{\Delta}$ and $UN$-rings}\label{sec3}

In this section, we seek some criteria under which matrix rings are \( U\sqrt{\Delta} \)-rings.

\begin{proposition}
If $M_n(R)$ is a $U\sqrt{\Delta}$-ring, then $C(R)$ is a local ring.
\end{proposition}

\begin{proof}
Since $M_n(R)$ is a $U\sqrt{\Delta}$-ring, Proposition \ref{center UJ} proposes that the center $C(M_n(R))$ is local. But it is well known that $C(M_n(R)) \cong C(R)$. Hence, $C(R)$ is a local ring.
\end{proof}

As a valuable consequence, we derive:

\begin{corollary}\label{mm}
Let $R$ be a commutative ring. Then, $M_n(R)$ is a $U\sqrt{\Delta}$-ring if and only if $R$ is a local ring.
\end{corollary}


Generally, for one-sided Artinian rings, we are able to present the following criterion.

\begin{theorem}\label{art}
Let $R$ be a left (resp., right) Artinian ring and $n$ a natural number. Then, $R$ is a $U\sqrt{\Delta}$-ring if and only if $R$ is a $UN$-ring, if and only if $M_n(R)$ is a $UN$-ring, if and only if $M_n(R)$ is a $U\sqrt{\Delta}$-ring.
\end{theorem}

\begin{proof}
If \( R \) is left Artinian, then it is well known that \( J(R) \) is nilpotent and thus nil. We show that, for all left Artinian rings, $UN$-rings and $U\sqrt{\Delta}$-rings are equivalent notions. Firstly, it is straightforward that every $UN$-ring is a $U\sqrt{\Delta}$-ring.

Now, to show the converse, assume that $R$ is a $U\sqrt{\Delta}$-ring. Since $R$ is left Artinian, we know that $$R/J(R) \cong \prod_{i=1}^k M_{m_i}(D_i),$$ where $k\in\mathbb{N}$ and each $D_i$ is a division ring. Thus, $$M_n(R)/J(M_n(R)) \cong M_n(R/J(R)) \cong \prod_{i=1}^{k}M_{nm_i}(D_i).$$

However, on the other hand, from Proposition~\ref{indecomposable} and Lemma~\ref{facto} we conclude that $R/J(R) \cong M_m(D)$, where $m\in\mathbb{N}$ and $D$ is a division ring. Since it follows from \cite[Proposition~0(a)]{vun} that $M_m(D)$ is a $UN$-ring and $J(R)$ is nil, it must be that $R$ is a $UN$-ring. Moreover, since it is principally known that \( R \) is left Artinian if and only if \( M_n(R) \) is left Artinian, it suffices to prove only that \( R \) is a $UN$-ring if and only if \( M_n(R) \) is a $UN$-ring. To that end, $R$ is a $UN$-ring if and only if $R/J(R)$ is a $UN$-ring, if and only if $k=1$ in accordance with Proposition~\ref{indecomposable}, if and only if
$M_n(R)/J(M_n(R))$ is a $UN$-ring, if and only if $M_n(R)$ is a $UN$-ring, because $J(M_n(R))$ is nil, as suspected.
\end{proof}


Given a ring $R$ and a central element $s$ of $R$, the $4$-tuple $\begin{pmatrix}
	R & R\\
	R & R
\end{pmatrix}$ becomes a ring with addition component-wise and with multiplication defined by
$$\begin{pmatrix}
	a_{1} & x_{1}\\
	y_{1} & b_{1}
\end{pmatrix}\begin{pmatrix}
	a_{2} & x_{2}\\
	y_{2} & b_{2}
\end{pmatrix}=\begin{pmatrix}
	a_{1}a_{2}+sx_{1}y_{2} & a_{1}x_{2}+x_{1}b_{2} \\
	y_{1}a_{2}+b_{1}y_{2} & sy_{1}x_{2}+b_{1}b_{2}
\end{pmatrix}.$$
This ring is hereafter denoted by $K_s(R)$.

\medskip

We are now prepared to prove the following.

\begin{proposition}
Suppose $R$ is a commutative ring and $s \in J(R)$. If $K_s(R)$ is a $U\sqrt{\Delta}$-ring, then $R$ is also a $U\sqrt{\Delta}$-ring.
\end{proposition}

\begin{proof}
Choose	$
	\begin{pmatrix}
		a & 0 \\
		0 & 0
	\end{pmatrix} \in K_s(R)
	$
to be a non-unit for some $a \in R$. Since $K_s(R)$ is a $U\sqrt{\Delta}$-ring, we can write
	\[
	\begin{pmatrix}
		a & 0 \\
		0 & 0
	\end{pmatrix} =
	\begin{pmatrix}
		u & 0 \\
		0 & 1
	\end{pmatrix}
	\begin{pmatrix}
		d & 0 \\
		0 & 0
	\end{pmatrix},
	\]
where
	$
	\begin{pmatrix}
		u & 0 \\
		0 & 1
	\end{pmatrix} \in U(K_s(R))$ and
	$\begin{pmatrix}
		d & 0 \\
		0 & 0
	\end{pmatrix} \in \sqrt{\Delta(K_s(R))}.
	$
But \cite[Lemma 4.1]{mdjo} gives us that $u \in U(R)$. Moreover, since
	$
	\begin{pmatrix}
		d^n & 0 \\
		0 & 0
	\end{pmatrix} \in \Delta(K_s(R))
	$
for some $n \in \mathbb{N}$, it follows again from \cite[Lemma 4.1]{mdjo} that $d^n \in \Delta(R)$. Hence, $d \in \sqrt{\Delta(R)}$, and so we obtain $a = ud$ with $u \in U(R)$ and $d \in \sqrt{\Delta(R)}$. Consequently, $R$ is a $U\sqrt{\Delta}$-ring, as promised.
\end{proof}


\section{Group Rings over $U\sqrt{\Delta}$-rings}\label{sec4}

Suppose that $G$ is an arbitrary group and $R$ is an arbitrary ring. As usual, $RG$ stands for the group ring of $G$ over $R$. It is known that the homomorphism $\varepsilon : RG\rightarrow R$, defined by $\varepsilon \left(\sum_{g\in G}a_{g}g\right)=\sum_{g\in G}a_{g}$, is called the {\it augmentation map} of $RG$ and its kernel, denoted by $\varepsilon (RG)$, is called the {\it augmentation ideal} of $RG$.

In \cite{kjun}, $UN$-group rings were characterized over the course of three pages. In what follows, we present a simpler and more concise classification of $UN$-group rings. Subsequently, using the same approach, we characterize $U\sqrt{\Delta}$-group rings.


\medskip

In this way, our first major result is the following one.

\begin{theorem}
Let \( F \) be a field and let \( G \) be a group. Then, the following two statements are equivalent:

(1) The group ring \( FG \) is a \( U\sqrt{\Delta} \)-ring.

(2) The group ring \( FG \) is local.
\end{theorem}

\begin{proof}
(2) \(\implies\) (1) This implication is immediate.
	
(1) \(\implies\) (2) Let \(\alpha \in FG\) be a non-unit element. Since \(FG\) is a \( U\sqrt{\Delta} \)-ring, there exist \( u \in U(FG) \) and \( z \in \sqrt{\Delta(FG)} \) such that $\alpha = uz$. Now, applying the augmentation map \(\varepsilon\) to both sides yields that
	\[
	\varepsilon(\alpha) = \varepsilon(u z) = \varepsilon(u) \varepsilon(z).
	\]
Because \(F\) is a field, it must be that \(\varepsilon(z) = 0\). Hence, $\varepsilon(\alpha) = 0$, which in turn ensures \(\alpha \in \varepsilon(FG)\). Therefore, every non-unit element lies in the ideal \(\varepsilon(FG)\), and consequently the set of non-units forms an ideal. This establishes that \(FG\) is local, as required.
\end{proof}

\begin{lemma}\label{1}
The inclusion $\varepsilon(\Delta(RG)) \subseteq \Delta(R)$ is always fulfilled.
\end{lemma}

\begin{proof}
Let \( d \in \Delta(RG) \) and \( u \in U(R) \). Since \(\varepsilon(u) = u\) and \(\varepsilon(U(RG)) \subseteq U(R)\), it follows directly that
	\[
	1 - u \varepsilon(d) = \varepsilon(1 - u d) \in \varepsilon(U(RG)) \subseteq U(R),
	\]
as needed.
\end{proof}

Our second main result is the following one.

\begin{theorem}\label{group ring}
Let $R$ be a ring and $G$ a non-trivial group.

(2) If $RG$ is a $U\sqrt{\Delta}$-ring, then $R$ is a $U\sqrt{\Delta}$-ring and $G$ is a $p$-group for a prime $p \in J(R)$.

(2) If $R$ is a $U\sqrt{\Delta}$-ring and $G$ is a locally finite $p$-group for a prime $p \in J(R)$, then $RG$ is a $U\sqrt{\Delta}$-ring.
\end{theorem}

\begin{proof}
(1) Since $\varepsilon(\Delta(RG)) \subseteq \Delta(R)$, it follows that $\varepsilon(\sqrt{\Delta(RG)}) \subseteq \sqrt{\Delta(R)}$. Next, we intend to show that $\varepsilon(RG) \subseteq J(RG)$. To that goal, suppose $\alpha \in \varepsilon(RG) \setminus J(RG)$. Then, there exists an element $\beta \in RG$ such that $1-\alpha\beta$ does not belong to $U(RG)$. But, since $RG$ is a $U\sqrt{\Delta}$-ring, there exist $u \in U(RG)$ and $j \in \sqrt{\Delta(RG)}$ such that $1-\alpha\beta = uj$. As $\varepsilon(RG)$ is an ideal, we perceive $1 = \varepsilon(u)\varepsilon(j)$, guaranteeing $$\varepsilon(u)^{-1} = \varepsilon(j) \in U(R) \cap \sqrt{\Delta(R)},$$ a contradiction. Hence, $\varepsilon(RG) \subseteq J(RG)$, as intended to prove. Now, consulting with \cite[Proposition~15(i)]{con}, $G$ is a $p$-group for a prime $p \in J(R)$. Moreover, since $RG/\varepsilon(RG) \cong R$, Proposition \ref{facto} assures that $R$ is a $U\sqrt{\Delta}$-ring.

(2) A consultation with \cite[Lemma~2]{zhouclean} leads to $\varepsilon(RG) \subseteq J(RG)$. But, since $RG/\varepsilon(RG) \cong R$, it follows from Proposition \ref{facto} that $RG$ is a $U\sqrt{\Delta}$-ring.
\end{proof}


It is known that every torsion locally solvable group is locally finite (see, e.g., \cite[Theorem~5.4.11]{Robinson}). Thus, the following assertion is an immediate consequence of Theorem~\ref{group ring}.

\begin{corollary}\label{loc}
Let $R$ be a ring and $G$ a locally solvable group. Then, $RG$ is a $U\sqrt{\Delta}$-ring if and only if $R$ is a $U\sqrt{\Delta}$-ring and $G$ is a locally finite $p$-group for a prime $p \in J(R)$.
\end{corollary}


Let $F$ be a field and let $G$ a locally solvable group. The above result could be applied to deduce that $FG$ is a $U\sqrt{\Delta}$-ring if and only if $FG$ is a $UN$-ring, if and only if $F$ has positive characteristic $p$ and $G$ is a locally finite $p$-group. More concretely, we establish the following.

\begin{corollary}
Let $F$ be an uncountable field and $G$ a group. Then, $FG$ is a $U\sqrt{\Delta}$-ring if and only if $FG$ is a $UN$-ring. In this case, $F$ has positive characteristic $p$ and $G$ is a locally finite $p$-group.
\end{corollary}

\begin{proof}
Assume that $FG$ is a $U\sqrt{\Delta}$-ring. To show that $FG$ is a $UN$-ring, it suffices to prove that $J(FG)$ is nil. Let $x \in J(FG)$, and let $H$ be the subgroup of $G$ generated by the support of $x$. Since $H$ is countable, $FH$ is a countably generated $F$-algebra. Applying \cite[Corollary~4.21]{lafc}, $J(FH)$ is nil. Now, since $x \in J(FG) \cap FH \subseteq J(FH)$, it follows that $x$ is nilpotent, as required.
\end{proof}

We finish the work in this paper with the following result.

\begin{proposition}
Let $R$ be an Artinian ring. Then, $RG$ is a $U\sqrt{\Delta}$-ring if and only if $(R/J(R))G$ is a $U\sqrt{\Delta}$-ring.
\end{proposition}

\begin{proof}
Since $R$ is an Artinian ring, we may follow \cite[Proposition 3]{con} to extract that $J(R)G\subseteq J(RG)$. On the other side, we know that the isomorphism $(R/J(R))G\cong RG/J(R)G$ is always true. Thus, the outcome follows immediately from Lemma~\ref{facto}, as expected.
\end{proof}


\bigskip

\noindent{\bf Acknowledgement.} This work is based upon research funded by Iran National Science Foundation (INSF) under project No. 40402401.

\medskip

\section*{Data availability}
No data was used for the research described in the article.

\section*{Declarations}
The authors declare no any conflict of interests while writing and preparing this manuscript.

\end{document}